\documentclass[11pt]{article}
\usepackage[a4paper]{anysize}\marginsize{3.5cm}{3.5cm}{1.3cm}{2cm}
\pdfpagewidth=\paperwidth \pdfpageheight=\paperheight
\usepackage{amsfonts,amssymb,amsthm,amsmath,eucal}
\usepackage{pgf}
\usepackage{bbm}

\usepackage{tikz} 
\usetikzlibrary{arrows}
\theoremstyle{plain}
\newtheorem{thm}{Theorem}[section]
\newtheorem{theorem}[thm]{Theorem}
\newtheorem{lemma}[thm]{Lemma}

\newtheorem{conjecture}[thm]{Conjecture}

\theoremstyle{definition}
\newtheorem{definition}[thm]{Definition}
\newtheorem{remark}[thm]{Remark}

\newtheorem{thevarthm}[thm]{\varthmname}

\newenvironment{varthm*}[1]{\trivlist\item[]{\bf #1.}\it}{\endtrivlist}

\renewcommand\ge{\geqslant}
\renewcommand\geq{\geqslant}
\renewcommand\le{\leqslant}
\renewcommand\leq{\leqslant}

\newcommand\be{\begin{eqnarray*}}
\newcommand\ee{\end{eqnarray*}}
\newcommand\Q{\mathbb Q}

\newcommand\Z{\mathbb Z}

\newcommand\K{\mathbb K}
\renewcommand\P{\mathbb P}
\newcommand\calo{{\mathcal O}}
\newcommand\cali{{\mathcal I}}
\newcommand\calf{{\mathcal F}}
\newcommand\call{{\mathcal L}}
\newcommand\calm{{\Lambda}}
\newcommand\newop[2]{\def#1{\mathop{\rm #2}\nolimits}}
\newop\log{log}
\newop\ord{ord}
\newop\Gal{Gal}
\newop\SL{SL}
\newop\Bl{Bl}
\newop\mult{mult}
\newop\mass{mass}
\newop\div{div}
\newop\codim{codim}
\newop\edim{edim}
\newop\vdim{vdim}
\newop\sing{sing}
\newop\Zeroes{Zeroes}
\newop\Tr{Tr}
\newop\Res{Res}
\newop\usc{usc}
\newcommand\eqnref[1]{(\ref{#1})}
\newop\HF{HF}
\newop\HP{HP}

\newcommand\eps{\varepsilon}

\renewcommand\chi{\HP}

\newcommand\lra{\longrightarrow}

\def\keywordname{{\bfseries Keywords}}%
\def\keywords#1{\par\addvspace\medskipamount{\rightskip=0pt plus1cm
   \def\and{\ifhmode\unskip\nobreak\fi\ $\cdot$
   }\noindent\keywordname\enspace\ignorespaces#1\par}}

\def\subclassname{{\bfseries Mathematics Subject Classification
(2010)}\enspace}
\def\subclass#1{\par\addvspace\medskipamount{\rightskip=0pt plus1cm
   \def\and{\ifhmode\unskip\nobreak\fi\ $\cdot$
   }\noindent\subclassname\ignorespaces#1\par}}

\definecolor{uuuuuu}{rgb}{0,0,1}
\definecolor{qqqqff}{rgb}{0,0,1}
\definecolor{xdxdff}{rgb}{0,0,1}

\def\endproof{\hspace*{\fill}\endproofsymbol\endtrivlist}

\def\endproofsymbol{\frame{\rule[0pt]{0pt}{6pt}\rule[0pt]{6pt}{0pt}}}

\newcommand\parag[1]{\paragraph*{#1.}\hskip-0.5em}

\begin{document}

\author{Thomas Bauer, Sandra Di Rocco, David Schmitz, Tomasz Szemberg, Justyna Szpond}
\title{On the postulation of lines and a fat line
}
\date{\today}
\maketitle
\thispagestyle{empty}

\begin{abstract}
   In the present note we show that the union of $r$ general lines
   and one fat line in $\P^3$ imposes independent conditions
   on forms of sufficiently high degree~$d$,
   where the bound is independent of the number of lines.
   This extends former results
   of Hartshorne and Hirschowitz on unions of general lines,
   and of Aladpoosh on unions of general lines and one double line.
\keywords{postulation problems, fat flats, Hilbert functions, Serre vanishing}
\subclass{14C20, 14F17, 13D40, 14N05}
\end{abstract}


\section{Introduction}
   Let $X\subset\P^n$ be a closed subscheme. The Hilbert function of $X$
   encodes a number of properties of $X$ and has been classically an object
   of vivid research in algebraic geometry and commutative algebra. We recall
   first the definition.

\begin{definition}[Hilbert function]
   The Hilbert function of a scheme $X\subset\P^n(\K)$ is
   $$\HF_X:\Z\ni d\to \dim_{\K}[S(X)]_d\in\Z,$$
   where $S(X)$ denotes the graded homogeneous coordinate ring of $X$.
\end{definition}
   It is well known that the Hilbert function becomes eventually (i.e., for large $d$)
   a polynomial. We denote this Hilbert polynomial of $X$ by $\chi_X$.
   Whereas the Hilbert polynomial can be (in principle) computed algorithmically,
   the Hilbert function is more difficult to compute.
   It may happen that the Hilbert function is equal to the Hilbert polynomial,
   for example for $\P^n$ we have $\HF_{\P^n}(d)=\chi_{\P^n}(d)$ for
   all $d\geq 0$, but this behaviour is rare. The next simplest
   behaviour occurs for subschemes with bipolynomial Hilbert function.
\begin{definition}[Bipolynomial Hilbert function]
   Following \cite{CCG10} we say that $X$ has
   a \emph{bipolynomial Hilbert function} if
   \begin{equation}\label{eq:HF bipoly}
      \HF_X(d)=\min\left\{\chi_{\P^n}(d),\chi_X(d)\right\}
   \end{equation}
   for all $d\geq 1$.
\end{definition}
   In other words, $X$ has a bipolynomial Hilbert function if
   $X\subset\P^n$ imposes the \emph{expected} number of conditions on forms of arbitrary
   degree $d\geq 1$.
   It is definitional that if $X$ consists of
   $q$ general points in $\P^n$, then its Hilbert function
   is bipolynomial. An analogous result for $X$ consisting
   of $r$ general lines in $\P^n$ with $n\geq 3$ has been
   proved by Hartshorne and Hirschowitz in \cite[Theorem 0.1]{HarHir81}.
   Recently Carlini, Catalisano
   and Geramita \cite{CCG13} showed that if $X$ consists of $r$ general lines and one general fat point,
   then, up to a short list of exceptions in $\P^3$, the Hilbert function of $X$
   is bipolynomial, see also \cite{AlaBal14} and \cite{Bal16}.

   Aladpoosh in \cite{Ala16} has proved recently that also a scheme consisting of
   $r$ general lines and one double line has (with the exception of one double line and two simple lines in $\P^4$ imposing
   dependent conditions on forms of degree $2$) a bipolynomial
   Hilbert function. She also conjectured \cite[Conjecture 1.2]{Ala16} that the same holds true
   for $r$ general lines and one fat flat of arbitrary dimension. In the present note we provide evidence
   supporting this conjecture for a fat line of arbitrary multiplicity $m$. Our main result is the
   following.
\begin{varthm*}{Main Theorem}
   Let $m\geq 1$ be a fixed integer.
   Then for $d\ge d_0(m):=3\binom{m+1}{3}$,
   the Hilbert function of a subscheme $X\subset\P^3$
   consisting of $r\geq 0$ general lines and one line of multiplicity $m$
   (i.e. defined by the $m$-th power of the ideal of a line) satisfies
   formula \eqnref{eq:HF bipoly}.
\end{varthm*}
   In other words, a general fat line and an arbitrary number $r$ of general lines
   with multiplicity $1$ impose independent conditions on forms of degree $d\geq d_0(m)$
   (see Theorem~\ref{thm:main rank}).

   It follows from the Serre Vanishing \cite[Theorem 1.2.6]{PAG} that for any subscheme
   $X\subset \P^n$, there exists a bound $d_0(X)$ such that $X$ imposes independent
   conditions on forms of degree $d\geq d_0(X)$. The point here is that
   we obtain an explicit bound that depends only on the multiplicity of the fat
   line but is independent of the number of reduced lines.

   We will set up the proof in a way which employs the general strategy of
   Hartshorne and Hirschowitz \cite{HarHir81}
   and Carlini, Catalisano and Geramita \cite{CCG13}.
   This amounts to work inductively by constructing a
   a suitable sequence of
   generic subschemes $Z_0,Z_1,\dots$,
   along with
   suitable
   specializations $Z_i'$ of $Z_i$.
   The starting scheme $Z_0$ consists of the lines in the theorem
   plus a number of generic points.
   The essential difficulty in this strategy lies in the
   question which kinds of intermediate schemes $Z_i$ to consider
   and which specializations $Z_i'$ to chose, in order for
   an inductive procedure to work.
   In our approach
   this is achieved by using
   intermediate schemes that
   contain,
   apart from disjoint lines and points, also
   crosses and so-called
   zig-zags (see Def.~\ref{def:zz}).


\section{Preliminaries and auxiliary results}
   We begin by recalling a formula for the number $c(n,m,d)$ of conditions which
   vanishing to order $m$ along a line in $\P^n$ imposes on forms
   of degree $d\ge m$:
   \begin{equation}\label{eq:conditions of fat line in Pn}
      c(n,m,d)=\frac{m(nd+2n+m-mn-1)}{n(n-1)}\binom{n+m-2}{m}.
   \end{equation}
   For a proof see e.g. \cite[Lemma 2.1]{DHST14}. Note that
   $$c(n,1,d)=d+1$$
   for all $n\geq 1$.

   In the next Lemma we present a useful formula relating some of numbers $c(n,m,d)$.

\begin{lemma}\label{lem:formula for cnmd}
   For all positive integers $n,m,d$ we have
   $$c(n,m,d)=c(n,m-1,d-1)+c(n-1,m,d).$$
\end{lemma}

\proof
   This is a straightforward computation.
\endproof

   In $\P^3$ the formula \eqref{eq:conditions of fat line in Pn} reduces to
   $$c(d,m)=c(3,m,d)=\frac16m(m+1)(3d+5-2m).$$

   Our approach to the Main Theorem uses the specialization method.
   This employs the Semi-Continuity Theorem
   \cite[Theorem III.12.8]{Har77}
   in the following way:

\begin{trivlist}\item[]\it
   Let $f:X\to B$ be a projective morphism of noetherian schemes
   and let $\calf$ be a coherent sheaf on $X$, flat over $B$.
   The vanishing $h^0(X_b,\calf_b)=0$ for some $b$ implies then the vanishing
   $h^0(X_{b'},\calf_{b'})=0$ for all $b'$ in a neighborhood of $b$.
\end{trivlist}
   In our situation, this means concretely that
   if $h^0(\P^n,\calo_{\P^n}(d)\otimes\cali_{Z_b})=0$
   for a (special) subscheme $Z_b$, then
   $h^0(\P^n,\calo_{\P^n}(d)\otimes\cali_{Z_{b'}})=0$
   for a (general) subscheme $Z_{b'}$ such that $Z_b$ and
   $Z_{b'}$ vary in a flat family over $B$.

\medskip
   We are going to use and generalize the notion of \emph{sundials}
   following the ideas of Carlini, Catalisano and Geramita, see
   \cite[Section 2]{CCG13} for definitions and motivations.

\begin{definition}[Sundials and crosses]
   A \emph{sundial} in $\P^n$ is the limiting subscheme obtained
   by a collision of two skew lines (hence spanning a $\P^3\subset\P^n$).
   It has a nonreduced structure in the collision point which can be thought of
   as a vector generating together with the plane spanned by the two
   intersecting lines the $\P^3$ mentioned above.

   A union of two lines in $\P^n$ intersecting in a single point is called a \emph{cross}.
   A cross is hence a sundial with the reduced structure.
\end{definition}

   Carlini, Catalisano and Geramita proved in \cite[Lemma 2.5]{CCG10} that there exists a flat family
   $g:W\to B$ of schemes in $\P^n$, with $n\geq  3$ such that a general member $W_{b'}\subset W$
   is a union of two disjoint lines, whereas the special fiber $W_b$ is a sundial.

   It is a crucial point in our proof of the Main Theorem to
   use a generalization of this idea, which uses \emph{zig-zags}
   in the following sense:

\begin{definition}[Zig-zag]\label{def:zz}
   A \emph{zig-zag} of length $z$ is the limiting subscheme obtained
   by a collision of an ordered set of $z$ general lines $L_1,L_2,\ldots,L_z$ in such a way,
   that the line $L_1$ intersects $L_2$, the line $L_2$ intersects $L_1$ and $L_3$
   and the intersection points are distinct, $L_3$ intersects $L_2$ and $L_4$ and
   the intersection points are again distinct, and so on, finally $L_{z-1}$ intersects
   $L_{z-2}$ and $L_z$ in two distinct points. The structure in the intersection
   points is the same as the structure of a sundial in the intersection point of its lines.
   A zig-zag of length $z$
   has thus $(z-1)$ singular points.

   A \emph{reduced zig-zag} is a zig-zag with reduced structure, i.e., no embedded points.
\end{definition}

   Figure \ref{fig:zz7} shows
   a zig-zag of length $7$. Note that the lines in the figure are all skew,
   there are no other intersection points but those indicated in this figure.
   The intersection points are embedded points with the
   structure of a scheme of length $2$ not contained in the plane generated
   by the intersecting lines. Note that a sundial is just a zig-zag of length $2$.
   A cross is a reduced zig-zag of length $2$.

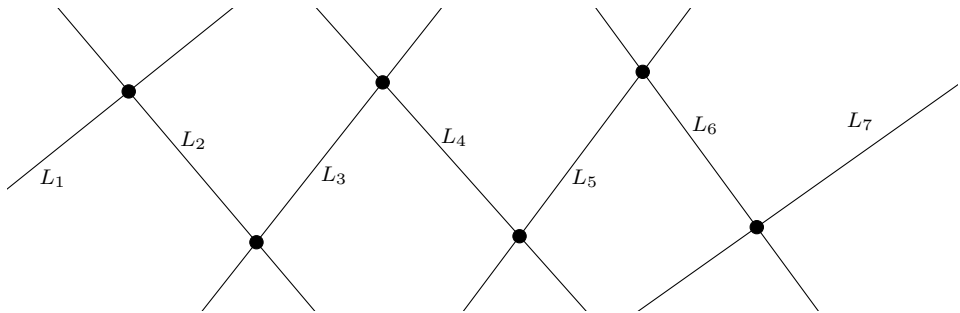
\begin{figure}[h]
\centering
\definecolor{qqqqff}{rgb}{0,0,0}
\begin{tikzpicture}[line cap=round,line join=round,>=triangle 45,x=1.0cm,y=1.0cm]
\clip(-1.5,-2.5) rectangle (11,1.5);
\draw [domain=-4.3:18.7] plot(\x,{(--0.872-2.*\x)/1.68});
\draw [domain=-4.3:18.7] plot(\x,{(-6.4296--2.12*\x)/1.66});
\draw [domain=-4.3:18.7] plot(\x,{(--7.9536-2.04*\x)/1.8});
\draw [domain=-4.3:18.7] plot(\x,{(-13.8856--2.18*\x)/1.62});
\draw [domain=-4.3:18.7] plot(\x,{(--15.1216-2.06*\x)/1.5});
\draw [domain=-4.3:18.7] plot(\x,{(--0.702--1.78*\x)/2.2});
\draw [domain=-4.3:18.7] plot(\x,{(-25.8--2.5*\x)/3.5});
\begin{scriptsize}
\draw[color=qqqqff] (-0.9,-0.75) node {$L_1$};
\draw [fill=qqqqff] (0.1,0.4) circle (2.5pt);
\draw[color=qqqqff] (0.95,-0.25) node {$L_2$};
\draw [fill=qqqqff] (1.78,-1.6) circle (2.5pt);
\draw[color=qqqqff] (2.8,-0.71) node {$L_3$};
\draw [fill=qqqqff] (3.44,0.52) circle (2.5pt);
\draw[color=qqqqff] (4.38,-0.21) node {$L_4$};
\draw [fill=qqqqff] (5.24,-1.52) circle (2.5pt);
\draw[color=qqqqff] (6.1,-0.75) node {$L_5$};
\draw [fill=qqqqff] (6.86,0.66) circle (2.5pt);
\draw[color=qqqqff] (7.68,-0.05) node {$L_6$};
\draw [fill=qqqqff] (8.36,-1.4) circle (2.5pt);
\draw[color=qqqqff] (9.72,0.01) node {$L_7$};
\end{scriptsize}
\end{tikzpicture}
\caption{A zig-zag of length $7$}
\label{fig:zz7}
\end{figure}

\begin{lemma}\label{lemma: zz}
   For an integer $z\geq 2$, there exists a flat family $\{X_{\lambda}\}$ of schemes in $\P^n$, with $n\geq 3$ such that a general
   member of $\{X_{\lambda}\}$ is a union of $z$ disjoint lines and the special fiber is a zig-zag of length $z$.
\end{lemma}

\proof
   The proof consists in a generalization of the argument
   in \cite[Lemma 2.5]{CCG10}.
\endproof

   Zig-zags are useful in our approach because of the following fact.

\begin{lemma}\label{lem:residual sundial}
   Let $S$ be a zig-zag of length $z$ in $\P^3$ formed by lines $L_1,\ldots,L_z$.
   Let $Q$ be a smooth quadric in $\P^3$ such that all singular points of $S$
   lie on $Q$ but none of the lines in $S$ is contained
   in $Q$. Then the colon ideal
   $$J=I_S:I_Q$$
   defines the reduced zig-zag $V(J)=L_1\cup\ldots\cup L_z$.
\end{lemma}

   Apart from semicontinuity, the residual exact sequence and the Castelnuovo inequality
   are key ingredients in the proof. We discuss them now.

\begin{definition}[Trace and residual scheme]
   Let $Y$ be a divisor of degree $e$ in $\P^n$ and let $Z\subset \P^n$ be a closed subscheme.
   Then the subscheme $Z''=\Tr_Y(Z)$ defined in $Y$ by the ideal
   $$I_{Z''/Y}=\left(I_Y+I_Z\right)/I_Y\subset\calo_Y$$
   is the \emph{trace of $Z$ on $Y$}.

   The colon ideal $I_{Z'}=(I_Z:I_Y)\subset \calo_{\P^n}$
   defines $Z'=\Res_Y(Z)$, the \emph{residual scheme of $Z$ with respect to $Y$}.
\end{definition}

   One has the following \emph{residual exact sequence}
   \begin{equation}\label{eq:residual sequence}
      0\lra\cali_{Z'}(-Y)\lra\cali_Z\lra\cali_{Z''/Y}\lra 0
      \,,
   \end{equation}
   where $\cali_W$ is the sheafification of the ideal $I_W$.
   Twisting \eqref{eq:residual sequence} by $\calo_{\P^n}(d)$ we get
   \begin{equation}\label{eq:residual twisted}
   0\lra \calo_{\P^n}(d-e)\otimes\cali_{\Res_Y(Z)}\lra \calo_{\P^n}(d)\otimes\cali_Z\lra \calo_Y(d)\otimes\cali_{\Tr_Y(Z)}\lra 0.
   \end{equation}
   Taking then the long cohomology sequence of \eqref{eq:residual twisted} we obtain
   the following
   statement, which is called the Castelnuovo inequality, see e.g. \cite[Lemma 3.3]{CCG11}.

\begin{lemma}[Castelnuovo inequality]\label{lem:Castelnuovo}
   Let $Y\subset \P^n$ be a divisor of degree $e$ and let $d\geq e$ be an integer.
   Let $Z\subset\P^n$ be a closed subscheme. Then
   \begin{equation}\label{eq:Castelnuovo}
      h^0(\P^n,\calo_{\P^n}(d)\otimes\cali_Z)\leq h^0(\P^n,\calo_{\P^n}(d-e)\otimes\cali_{\Res_Y(Z)})+h^0(Y,\calo_Y(d)\otimes\cali_{\Tr_Y(Z)/Y}).
   \end{equation}
\end{lemma}

   We call the space $H^0(\P^n,\calo_{\P^n}(d-e)\otimes\cali_{\Res_Y(Z)})$ the \emph{residual linear system} of
   $H^0(\P^n,\calo_{\P^n}(d)\otimes\cali_Z)$ with respect to $Y$ and
   $H^0(Y,\calo_Y(d)\otimes\cali_{\Tr_Y(Z)/Y})$ the \emph{trace linear system} of
   $H^0(\P^n,\calo_{\P^n}(d)\otimes\cali_Z)$ on $Y$.


\section{Nonspeciality of certain linear series on $\P^1\times\P^1$}

   In the proof of the Main Theorem we will consider trace linear systems
   on a smooth quadric in $\P^3$. This section serves as a preparation
   of relevant results on linear systems on a smooth quadric in $\P^3$
   identified with $\P^1\times\P^1$. Special linear systems with general points of multiplicity at most $3$
   on $\P^1\times\P^1$ have been classified by Lenarcik in \cite{Len11}.
   Here we recall a part of \cite[Theorem 2]{Len11} relevant in our situation.

\begin{lemma}\label{lem:system 1 on P1xP1}
   Let $Z$ be the fat point scheme in $\P^1\times\P^1$ defined by the ideal
   $$I_Z=I(P_1)^2\cap\ldots\cap I(P_p)^2\cap I(Q_1)\cap\ldots\cap I(Q_q),$$
   where $P_1,\ldots,P_p,Q_1,\ldots,Q_q$ are general points in $\P^1\times\P^1$.
   Let $0\leq a\leq b$ be non-negative integers. The linear system
   $$H^0(\P^1\times\P^1, \calo_{\P^1\times\P^1}(a,b)\otimes \cali_Z)$$
   is special if and only if one of the following cases holds
   \begin{itemize}
      \item $a=0$ and $p+2q\leq b$,
      \item $a=2$, $p=0$, $b=q-1$ and $q$ is odd.
   \end{itemize}
\end{lemma}

   Using this result, we prove now
   an auxiliary postulation statement
   for higher multiplicities:

\begin{lemma}\label{lem:system 2 on P1xP1}
   Given an integer $m\geq 2$ let $k$ be an integer with $k\geq\binom{m+1}{3}$.
   Then $2$ general points $P_1$, $P_2$ taken with multiplicity $m$ impose
   independent conditions on linear systems on $\P^1\times\P^1$ of bidegree $(a,b)$
   if $a\leq b$ and $a\geq k-1$ and $b\geq 3k$.
\end{lemma}

\proof
   For $m=2$ the assertion for arbitrary $k\geq\binom{m+1}{3}$ follows from Lemma \ref{lem:system 1 on P1xP1}.
   We proceed by induction on $m$ and $k$. Let $m$ and $k\geq\binom{m+1}{3}$ be fixed
   and assume that the assertion holds for all $m'<m$ and $k'$.
   Let $s=(a+1)(b+1)-2\binom{m+1}{2}$ and let $Q_1,\ldots,Q_s$ be $s$ general
   points in $\P^1\times\P^1$. It is enough to show that there is
   no divisor of bidegree $(a,b)$ which passes with multiplicity $m$
   through the points $P_1$, $P_2$ and passes through $Q_1,\ldots,Q_s$.
   It suffices to prove this claim for a particular position of points
   $Q_1,\ldots,Q_s$.

   To this end let $C$ be a smooth curve of bidegree $(1,1)$ passing through $P_1$ and $P_2$.
   Thus $C$ is a smooth rational curve. Let $t=a+b-2m+1$. By above assumptions this
   is a non-negative integer. We specialize now the points $Q_1,\ldots,Q_t$
   onto the curve $C$ leaving the points $Q_{t+1},\ldots,Q_s$ as general
   points on $\P^1\times\P^1$, so that they do not lie on $C$ in particular.
   Assume to the contrary that there is a divisor $\Gamma$ such that
   $\mult_{P_i}\Gamma\geq m$ for $i=1,2$ and $\mult_{Q_j}\Gamma\geq 1$ for $j=1,\ldots,s$.
   Then $C$ must be a component of $\Gamma$, because $(\Gamma\cdot C)=a+b$ but the trace
   of $\Gamma$ on $C$ has at least $2$ points of multiplicity $m$ and another $t$ points
   with $2m+t= a+b+1$. The residual divisor $\Gamma'=\Gamma-C$ has bidegree $(a-1,b-1)$
   and passes through the points $P_1$ and $P_2$ with multiplicity $m-1$
   and also passes through the points $Q_{t+1},\ldots,Q_{s}$. Since $s-t=ab-2\binom{m}{2}$,
   the existence of $\Gamma'$ is excluded by our induction assumption.

   Thus we are done with the proof of the Lemma.
\endproof


\section{The proof of the Main Theorem}

   In this section we will prove the Main Theorem, which
   is equivalent to the following statement.

\begin{theorem}[Maximal rank property]\label{thm:main rank}
   For a subscheme $W\subset\P^n$ consisting of a general line of multiplicity $m$
   and an arbitrary number $r$ of general lines, for all $d\geq d_0(m)=3\binom{m+1}{3}$,
   the restriction map
   $$H^0(\P^n,\calo_{\P^n}(d))\to H^0(W,\calo_W(d))$$
   has maximal rank.
\end{theorem}

   As pointed out in the introduction, we will employ the general
   strategy of Hartshorne and Hirschowitz \cite[Theorem 1.1]{HarHir81}.
   Specifically, we will proceed inductively along a
   suitable sequence of subschemes $Z_0,Z_1,\dots$,
   for which we choose suitable specializations $Z_0',Z_1',\dots$.
   While we can start with a subscheme $Z_0$ consisting of
   general lines, a fat line and points,
   it is a major obstacle that
   it seems insufficient
   to use only these kinds of schemes during the whole induction process.
   Our idea is to instead allow intermediate schemes
   $Z=Z(m,r,s,q,z)$ consisting of
   one general line of multiplicity $m$, $r$ general lines,
   $s$ general crosses, $q$ general points and a reduced zig-zag of length $z$
   (along with
   particular specializations $Z'$ of $Z$, which will be introduced
   in Definition~\ref{def:specialization}).

   We now set up some notation that will be useful for the remainder of the paper.
   We denote by
   $$\call(k,\eps;m,r,s,q,z)=\call(d;Z)=H^0(\P^3,\calo_{\P^3}(d)\otimes\cali_Z)$$
   the linear system
   of polynomials in $\P^3$ of degree $d=3k+\eps$, with $\eps\in\{0,1,2\}$ vanishing along the subscheme $Z$.

   Similarly we will write
   $$\calm((a,b);p,p_d,p_m,m)=\calm((a,b);\Omega)=H^0(\P^1\times\P^1,\calo_{\P^1\times\P^1}(a,b)\otimes\cali_{\Omega})$$
   to indicate the linear system
   on $\P^1\times\P^1$ of polynomials of bidegree $(a,b)$ vanishing along the subscheme
   $\Omega=\Omega(p,p_d,p_m,m)$ consisting of $p$ general points, $p_d$ general double
   points and $p_m$ general points of multiplicity $m$. In our considerations $p_m$
   is either $0$ or $2$, depending on whether we specialize the fat line onto the
   quadric or not.

   Given $m\geq 1$ and $d\geq d_0(m)=3\binom{m+1}{3}$ there exist unique integers
   $r(d,m)\ge 0$ and $0\leq q(d,m)\leq d$ such that
   \begin{equation}\label{eq:sum}
      \chi_{\P^3}(d)=c(d,m)+r(d,m)(d+1)+q(d,m).
   \end{equation}
   So $\chi_{\P^3}(d)$ is the virtual number
   of conditions that one $m$-fold line,
   $r(d,m)$ generic ordinary lines, and $q(d,m)$ generic points impose.

\begin{remark}\label{rem:formulas r and q}
   Concretely, we have
   $$r(d,m)=\left\lfloor \frac{1}{d+1}\left(\binom{d+3}{3}-\frac16m(m+1)(3d+5-2m)\right)\right\rfloor$$
   and
   $$q(d,m)=\binom{d+3}{3}-\frac16m(m+1)(3d+5-2m)-(d+1)r(d,m).$$
   In particular,
   \begin{itemize}
   \item for $d=3k$
      $$r(d,m)=\frac32k^2+\frac52k+1-\binom{m+1}{2}\;\mbox{ and }\; q(d,m)=2\binom{m+1}{3},$$
   \item for $d=3k+1$
      $$r(d,m)=\frac32k^2+\frac72k+2-\binom{m+1}{2}\;\mbox{ and }\; q(d,m)=2\binom{m+1}{3},$$
   \item for $d=3k+2$
      $$r(d,m)=\frac32k^2+\frac92k+3-\binom{m+1}{2}\;\mbox{ and }\; q(d,m)=k+1+2\binom{m+1}{3}.$$
   \end{itemize}
\end{remark}

   The following theorem
   (to be proved in Subsection~\ref{sect:proof})
   implies the Main Theorem.
   \begin{theorem}\label{th:sequence}
      Let $d \ge d_0(m)=3\binom{m+1}{3}$ and let
      $Z = Z(m,r(m,d),0,q(m,d),0)$, or $Z=Z(m,r(m,d)+1,0,0,0)$.
      Let further be $Q$ some smooth quadric.
      Then there exists a sequence $Z=Z_0,  Z_1,\dots,Z_u$ of schemes
      $Z_i=Z(m_i, r_i ,s_i ,q_i ,z_i)$  together with specializations $Z'_i$
      such that the following hold for each $i= 0,\dots, u-1$
      \begin{enumerate}
      \item[\rm(1)] $Z_{i+1} = Res_Q(Z'_i) $
      \item[\rm(2)] $h^0(Q, \mathcal O_Q(d-2i) \otimes I_{Tr_Q(Z'_i)} ) = 0$
      \end{enumerate}
   and such that $Z_u$ satisfies the conditions
   \begin{itemize}
      \item[\rm(i)] $Z_u = Z(m_u, r(m_u, d-2u), 0, q(m_u, d- 2u), 0 )$, or \newline $Z_u =
      Z(m_u, r(m_u, d-2u)+1, 0,0, 0 )$
      \item[\rm(ii)] $d-2u \ge d_0(m_u)$
      \item[\rm(iii)] $m_u\in\left\{m-1,m-2,1,0\right\}$
   \end{itemize}
   \end{theorem}

\begin{proof}[Proof of Theorem~\ref{thm:main rank}]
   We proceed by induction on $m$.
   The base case $m=1$ has been proved for all $d\geq 0=d_0(1)$ in \cite{HarHir81}
   and the base case $m=2$ by Aladpoosh \cite{Ala16} for all $d\geq 2=d_0(2)$.

   Let now $m\ge 3$.
   For $d\ge d_0(m)$ it suffices to prove the bijectivity of the restriction map in the case of
   schemes $Z=Z(m,r(d,m),0,q(d,m),0)$, and the injectivity in the case of
   schemes $Z=Z(m,r(m,d)+1,0,0,0)$.
   This amounts in either case to proving the identity
   $$
      h^0(\call(d;Z))=0
      \,.
   $$
   Now, Theorem \ref{th:sequence} together with Castelnuovo's inequality yields
   \begin{eqnarray*}
      h^0(\call(d;Z)) &\le& h^0(\call(d-2u; Z_u)) + \sum_{i=1}^{u-1} h^0(Q, \mathcal O_Q(d-2i) \otimes I_{Tr_Q(Z'_i)} )\\
      &=& h^0(\call(d-2u; Z_u))
      \,,
   \end{eqnarray*}
   but the latter must be zero since $Z_n$ satisfies the induction hypothesis, again by Theorem \ref{th:sequence}.
\end{proof}


\subsection{Proof of Theorem \ref{th:sequence}}\label{sect:proof}

   In order to prove Theorem \ref{th:sequence}, we will need the next lemma describing which
   schemes result from certain specializations.

   \begin{definition}\label{def:specialization}
   Let $Q$ be a smooth quadric in $\P^3$. We denote by
   $R(\delta,\ell,\ell_s,\ell_z,t,t_s,t_z)$ the specialization $Z'$ of   $Z=Z(m,r,s,q,z)$ given by
   assuming  the following
   lines to be disjoint lines belonging to the same ruling of $Q$:
   \begin{itemize}
      \item $\delta$ $m$-fold lines (here $\delta$ will be either $0$ or $1$);
      \item $\ell$ ordinary lines;
      \item $\ell_s$ lines from $\ell_s$ crosses (one line from each cross);
      \item $\ell_z=\lfloor \frac{z}{2}\rfloor$ lines from the reduced zig-zag of length $z$,
   \end{itemize}
    and assuming furthermore
    \begin{itemize}
    \item $t$ among the $q$ points to be general points on $Q$,
    \item $2t_s$ of the $r$ lines to form $t_s$ sundials whose intersection with
    $Q$ is a zero-dimensional scheme containing the singular points of the sundials,
    \item  $t_z+1$ of the lines to form one zig-zag whose zero-dimensional intersection
    with $Q$ contains all $t_z$ singular points.
    \end{itemize}
\end{definition}

\begin{lemma}\label{lem:reduction}
   Let $Z'$ be the specialization $R(\delta,\ell,\ell_s,\ell_z,t,t_s,t_z)$ of the scheme
   $Z=Z(m,r,s,q,z)$.
   Then
      $$\Res_Q(Z')=Z(m-\delta,r-\ell+\ell_s+(z-\ell_z)-2t_s-(t_z+1),s-\ell_s+t_s,q-t,t_z+1)$$
   and
   $$\Tr_Q(Z')=D+\Omega(2r-2\ell-2\ell_z-3\ell_s-2t_s-2t_z+t+4s+z+\gamma,t_s+t_z,2-2\delta,m),$$
   where
   $D$ is a divisor on $Q$ consisting of $\delta$ lines,
   where $\delta\in\{0,1\}$,
   of multiplicity $m$
   and $\ell+\ell_s+\ell_z$
   reduced lines, all contained in the same ruling on $\Q$.
   Here
   $\gamma=\left\{\begin{array}{cc}
      0,&\mbox{if }\ell_z=0,\\
      1,&\mbox{if }\ell_z>0
      \end{array}\right.$.
\end{lemma}

   Now we turn to the proof of Theorem \ref{th:sequence}.

   The particular sequence of subschemes differs according to the divisibility of $d$ by $3$.
   In order to simplify notation we denote the relevant linear series by
   \begin{eqnarray*}
   B(k,\eps,m) &=& \call(k,\eps;m,r(3k+\eps,m),0,q(3k + \eps,m),0)     \\
   I(k,\eps,m) &=&   \call(k,\eps;m,r(3k+\eps,m)+1,0,0,0)
   \end{eqnarray*}
   The following table shows for each case the length and the final element of the sequence that we will
   construct in the sequel.

\begin{center}
\def\arraystretch{1.5}
\begin{tabular}{ccc}
   \hline
   For & a sequence of length   & yields\\
   \hline
   $B(k,0,m)$ & $1$ & $B(k-1,1,m-1)$\\
   $B(k,1,m)$ & $2$ & $B(k-1,0,m-1)$\\
   $B(k,2,m)$ & $1$ & $B(k,0,m-1)$\\
   \hline
   $I(k,0,m)$ & $2$ & $I(k-2,2,m-2)$\\
   $I(k,1,m)$ & $1$ & $I(k-1,2,m-1)$\\
   $I(k,2,3\ell)$ & $3\ell-1$ & $B(k-2\ell+1,1,1)$\\
   $I(k,2,3\ell+1)$ & $3\ell+1$ & $B(k-2\ell,0,0)$\\
   $I(k,2,3\ell+2)$ & $3\ell+1$ & $B(k-2\ell,0,1)$\\
   \hline
\end{tabular}
\end{center}

\subsubsection{The bijective cases}
   With $d=3k+\eps$, the initial system in every case here is
   $$\call(k,\eps;m,r(3k+\eps,m),0,q(3k+\eps,m),0).$$


    \parag{Case $B(k,0,m)$}
    We only specialize once, and we pick
    $$Z'_0=R(1,2k+1-m,0,0,m(m-1),0,0).$$
    By Lemma \ref{lem:reduction}, we obtain the trace system
    $$
       H^0(\mathcal O_Q(d) \otimes I_{Tr_Q(Z')})=\calm((d,d-(2k+1));2r-2(2k+1-m)+m(m-1),0,0,m)
     $$
     which is of virtual dimension
     $$
        (3k+1)k-(2r(3k,m)-2(2k+1-m)+m(m-1))=(3k+1)k-(3k+1)k=0.
     $$
     By Lemma \ref{lem:system 1 on P1xP1}, this system is non-special, so its actual dimension is also zero.
     This shows that condition~(2) in Theorem~\ref{th:sequence} is fulfilled.
    The residual system is
    \begin{eqnarray*}
    \call_1&=&\call(k-1,1;m-1,r(3(k-1)+1,m-1),0,q(3(k-1)+1,m-1),0)\\
    &=& B(k-1,1,m-1)
    \end{eqnarray*}
    by Lemma \ref{lem:reduction}.
    Note that the subscheme $Z_1:=\Res_Q(Z_0')$ then satisfies conditions~(i)--(iii)
    of Theorem~\ref{th:sequence}.


    \parag{Case $B(k,1,m)$}
    In this case we use two specializations. First set
    $$Z_0'=R(1,2k+1-m,0,0,m(m-1),2k,0),$$
    resulting in
    $$\call_1=\call(k-1,2;m-1,\frac32k^2-\frac52k+1-\frac12m^2+\frac12m,2k,\frac13m^3-m^2+\frac23m,0)$$
    and
    $$\calm_1=\calm(k,3k+1;3k^2-k+2,2k,0,m),$$
    which system is zero-dimensional.
    Then we set
    $$Z_1'=R(0,1,2k,0,0,0,0)$$
    and obtain the residual system
    $$\call_2=\call(k-1,0;m-1,r(3(k-1),m-1),0,q(3(k-1),m-1),0)= B(k-1,0,m-1),$$
    and the trace system
    $$\calm_2=\calm(k-2,3k-1;3k^2-3k-m^2+m,0,2,m-1),$$
    with $h^0(\calm_2)=0$.


    \parag{Case $B(k,2,m)$}
    In this case we use the specialization
    $$Z_0'=R(1,2k+2-m,0,0,k+1+m(m-1),0,0).$$
    We obtain
    $$\call_1=\call(k,0;m-1,r(3k,m-1),0,q(3k,m-1),0)=B(k,0,m-1)$$
    and
    $$\calm_1=\calm(k,3k+2;3k^2+6k+3,0,0,m)$$
   which is of dimension 0.


\subsubsection{The injective cases}

   With $d=3k+\eps$, the initial state in every case now is
   $$\call(k,\eps;m,r(3k+\eps,m)+1,0,0,0).$$


   \parag{Case $I(k,0,m)$}
   We have $\call_0=\call(k,0,m,r(3k,m)+1,0,0,0)$ so that
   $$\vdim(\call_0)=-3k-1+\frac13m(m-1)(m+1)<0$$
   for $d = 3k \ge d_0(m) = 3\binom{m+1}{3}$.

   We apply the specializations
   \begin{eqnarray*}
      Z_0'&=&R(1,2k+1-m,0,0,0,0,m(m-1)-2)\\
      Z'_1&=&R(1,2k+1-m-(\frac12 m(m-1)-1),0,\frac12 m(m-1)-1,0,0,0)
   \end{eqnarray*}
   By Lemma \ref{lem:reduction}  the trace systems are
   \begin{eqnarray*}
      \calm_1 &=& \calm((3k, k-1); 2(r(3k,m)+1-(2k+1-m)),m(m-1)-2,0,m) \\
      \calm_2 &=&\calm((3k-2, k-2); 3k^2-3k+2-2m^2+4m, 0, 0, m-1)
   \end{eqnarray*}
   It is easy to see that both of these have non-positive virtual dimensions for $d\ge d_0(m)$,
   and thus actual dimension zero.

   Note also that we have the identity
   $$r(3k,m)+1-(2k+1-m)-(2k+1-m)=r(3(k-2)+2,m-2)+1.$$
   The final residual system thus is
   $$\call_2=\call(k-2,2,m-2,r(3(k-2)+2,m-2)+1,0,0,0) = I(k-2,2,m-2).$$


   \parag{Case $I(k,1,m)$}
   Here $\call_0=\call(k,1,m,r(3k+1,m)+1,0,0,0)$, which has virtual dimension
   $$\vdim(\call_0)=-3k-2+\frac13m(m-1)(m+1)<0.$$
   We apply the specialization
   $$Z'_0=R(1,2k+2-m,0,0,0,0,0)$$
   which by the identity
   $$
      r(3k+1,m)+1-(2k+2-m)=r(3(k-1)+2,m-1)+1
   $$
   yields
   $$\call_1=\call(k-1,2,m-1,r(3(k-1)+2,m-1)+1,0,0,0) = I(k-1,2,m-1)$$
   as the residual system and
   $$\calm((k-1, 3k+1); 3k^2+3k+2-m^2+m, 0, 0, m)$$
   as the trace system. Its virtual dimension is
   $$\vdim(\calm_1)=-k-2+m^2-m<0,$$
   so $h^0(\calm_1)=0$.


   \parag{Case $I(k,2,m)$}
   This is the most difficult case -- it requires the use of zig-zags, and
   the specializations and their number depend on the multiplicity $m$ of the fat line
   as well as
   on the divisibility of $m$ by $3$. In this step, additionally, the reduction
   goes to one of the bijectivity cases.

   $\call_0=\call(k,2,m,r(3k+2,m)+1,0,0,0)$ and
   $$\vdim(\call_0)=-2k-2+\frac13m(m-1)(m+1)<0.$$
   In each case the first specialization will be
   $$Z_0'=R(1,2k+2-m,0,0,0,0,k+m(m-1)+1).$$
   Define further for $p=2,\dots,m-1$
   $$Z'_{p-1}=R(1,2k+2-m-\left\lfloor{\frac{p-1}{3}}\right\rfloor-\left\lfloor{\frac{t_{z_{p-1}}+1}{2}}\right\rfloor,0,\left\lfloor{\frac{t_{z_{p-1}}+1}{2}}\right\rfloor,0,0,t_{z_p}),$$
   where
   \[
   t_{z_p}=
   \begin{cases}
   k+pm(m-p)+\frac13p(p-1)(p+1)-2p+1 &\mbox{if } p \equiv 1,2\\
   pm(p-m)+\frac13p(p-1)(p+1)-2p+2\frac{p}{3} & \mbox{if } p \equiv 0
   \end{cases} \pmod{3}.
   \]

   Note that $t_{z_p}$ is chosen in a way that guarantees the corresponding
   trace systems to have virtual dimension zero, and thus actual dimension zero.

   \parag{Subcase $I(k,2,m=3\ell)$}
   In this case we consider the sequence $Z_0, Z_1,\dots, Z_{m-2}$ defined above
   and use as a final step $Z_{m-1}=\mbox{Res}_Q(Z')$ for
   $$Z'=R(1,2k+2-m-(\ell-1)-\left\lfloor{\frac{t_{z_{m-2}}+1}{2}}\right\rfloor+1,0,\left\lfloor{\frac{t_{z_{m-2}}+1}{2}}\right\rfloor,0,0,0)$$
   The final residual system is
   \begin{eqnarray*}
   \call_{m-1}&=&\call(k-2\ell+1,1,1,r(3k+2,3\ell)+\frac{21}{2}\ell-\frac{23}{2}\ell+3-6k\ell+2k,0,0,0).
   \end{eqnarray*}
   Since
   $$r(3k+2,3\ell)+\frac{21}{2}\ell-\frac{23}{2}\ell+3-6k\ell+2k=r(3(k-2\ell+1)+1,1)$$
   and $q(3(k-2\ell+1)+1,1)=0$ we have
   \begin{eqnarray*}
   \call_{m-1}&=&B(k-2\ell+1,1,1).
   \end{eqnarray*}
   The final trace system is
   $$\calm_{m-1}=\calm(k-2\ell, 3k-6\ell+6, 2r(3k+2,3\ell)-12k\ell-16\ell+21\ell^2+3k+3-9\ell^3, 0, 0, 2),$$
   which has  virtual dimension $-2k-2+\frac13m(m-1)(m+1)<0$.

   \parag{Subcase $I(k,2,m=3\ell+1)$}
   Consider the sequence $Z_0, Z_1,\dots, Z_{m-1}$ defined above
   and use as a final step $Z_m=\mbox{Res}_Q(Z')$ for
   $$Z'=R(1,2k+2-m-\ell-\left\lfloor{\frac{t_{z_{m-1}}+1}{2}}\right\rfloor+1,0,\left\lfloor{\frac{t_{z_{m-1}}+1}{2}}\right\rfloor,0,0,0)$$
   The final residual system is
   $$\call_m=\call(k-2\ell,0,0,r(3k+2,3\ell+1)+\frac{21}{2}\ell^2-\frac12\ell-1-6k\ell-2k,0,0,0)$$
   which thanks to the identities
   $$r(3k+2,3\ell+1)+\frac{21}{2}\ell^2-\frac12\ell-1-6k\ell-2k=r(3(k-2\ell),0)$$
   and $q(3(k-\ell),0)=0$
   equals the system $B(k-2\ell,0,0)$, as required.
   The final trace system is
   $$\calm_m=\calm(k-2\ell-1, 3k-6\ell+2, 2r(3k+2,3\ell+1)-12k\ell-4k-9\ell^3+12\ell^2+\ell-2, 0, 0, 1)$$
   Also in this case we have
   $$\vdim(\calm_m)=-2k-2+\frac13m(m-1)(m+1)<0.$$

   \parag{Subcase $I(k,2,m=3\ell+2)$}
   Use as in the first subcase the sequence $Z_0, Z_1,\dots, Z_{m-2}$ defined above
   and use as a final step $Z_{m-1}=\mbox{Res}_Q(Z')$ for
   $$Z'=(1,2k+2-m-\ell-\left\lfloor{\frac{t_{z_{m-2}}+1}{2}}\right\rfloor+1,0,\left\lfloor{\frac{t_{z_{m-2}}+1}{2}}\right\rfloor,0,0,0)$$
   The final residual system is
   $$\call_{m-1}=\call(k-2\ell,0,1,r(3k+2,3\ell+2)+\frac{21}{2}\ell^2+\frac52\ell-6k\ell-2k,0,0,0)$$
   with
   $$r(3k+2,3\ell+2)+\frac{21}{2}\ell^2+\frac52\ell-6k\ell-2k=r(3(k-2\ell),1)$$
   and $q(3(k-2\ell),1)=0$.
   The final trace system is
   $$\calm_{m-1}=\calm(k-2\ell-1, 3k-6\ell+2, 2r(3k+2,3\ell+2)-12k\ell-4k-9\ell^3+3\ell^2-2\ell, 0, 0, 2).$$
   Its dimension is zero since
   $$\vdim(\calm_{m-1})=-2k-2+\frac13m(m-1)(m+1)<0$$
    \endproof



\section{Final remarks}

   We have developed a software to handle calculations necessary here.
   The software proved indispensable in order to manipulate sets of data
   and to discover general patterns leading to suitable specializations.
   Using this software we were not able to find any systems in the range
   $d<d_0(m)$ for which the maximal rank statement in Theorem \ref{thm:main rank}
   would fail.
   We therefore expect that the statement holds in these cases as well:

\begin{conjecture}[Maximal Rank Conjecture]\label{conj:max rank}
   The restriction maps in Theorem~\ref{thm:main rank}
   have maximal rank for all $d\geq 1$.
\end{conjecture}
   We hope that with some modifications, the software mentioned above might prove
   useful in similar situations, in particular might help to advance
   towards the proof of Aladpoosh's Conjecture. We also expect that our results
   can be generalized to projective spaces of arbitrary dimension. This
   is a subject of our current research.

\paragraph*{Acknowledgement.}
   This project has been started during the mini-workshop ''Arrangements of Subvarieties, and their Applications''
   held in Mathematischem Forschungsinstitut Oberwolfach February 29 - March 5, 2016. We are grateful to MFO
   for providing perfect working conditions. Part of the paper has been written up during the visit
   of the last two authors at KTH. We thank KTH for providing financial support and stimulating
   working atmosphere. The final reductions have been discovered at the Karma coffee place in Krak\'ow.
   It is a pleasure to acknowledge support and understanding of the Karma team.

   Research of DS was supported by DFG research fellowship SCHM 3223/1-1.
   Research of TS and JS was partially supported by National Science Centre, Poland, grant
   2014/15/B/ST1/02197.



\bigskip \small

\bigskip

   Thomas Bauer,
   Fach\-be\-reich Ma\-the\-ma\-tik und In\-for\-ma\-tik,
   Philipps-Uni\-ver\-si\-t\"at Mar\-burg,
   Hans-Meer\-wein-Stra{\ss}e,
   D-35032~Mar\-burg, Germany

\nopagebreak
   \textit{E-mail address:} \texttt{tbauer@mathematik.uni-marburg.de}

\bigskip

   Sandra Di Rocco,
   Department of Mathematics, KTH, 100 44 Stockholm, Sweden.

\nopagebreak
   \textit{E-mail address:} \texttt{dirocco@math.kth.se}

\bigskip

\nopagebreak
 David Schmitz,
   Fach\-be\-reich Ma\-the\-ma\-tik und In\-for\-ma\-tik,
   Philipps-Uni\-ver\-si\-t\"at Mar\-burg,
   Hans-Meer\-wein-Stra{\ss}e,
   D-35032~Mar\-burg, Germany

\nopagebreak
   \textit{E-mail address:} \texttt{schmitzd@mathematik.uni-marburg.de}

\bigskip
   Tomasz Szemberg,
   Department of Mathematics, Pedagogical University of Cracow,
   Podchor\c a\.zych 2,
   PL-30-084 Krak\'ow, Poland

\nopagebreak
   \textit{E-mail address:} \texttt{tomasz.szemberg@gmail.com}

\bigskip
   Justyna Szpond,
   Department of Mathematics, Pedagogical University of Cracow,
   Podchor\c a\.zych 2,
   PL-30-084 Krak\'ow, Poland

\nopagebreak
   \textit{E-mail address:} \texttt{szpond@up.krakow.pl}



\begin{thebibliography}{99}\footnotesize

\bibitem{Ala16}
   Aladpoosh, T.:
   Postulation of general lines and one double line in $\P^n$ in view of general lines and
   one multiple linear space.
   arXiv: 1606.02974

\bibitem{AlaBal14}
   Aladpoosh, T., Ballico, E.:
   Postulation of disjoint unions of lines and a multiple point.
   Rend. Semin. Mat. Univ. Politec. Torino 72 (2014), no. 3-4, 127--145

\bibitem{Bal16}
   Ballico, E.:
   Postulation of Disjoint Unions of Lines and a Multiple Point, II.
   Mediterr. J. Math. 13 (2016), no. 4, 1449--1463

\bibitem{CCG10}
   Carlini, E., Catalisano, M. V., Geramita, A. V.:
   Bipolynomial Hilbert functions.
   J. Alg. 324 (2010), 758--781

\bibitem{CCG11}
   Carlini, E., Catalisano, M. V., Geramita, A. V.:
   $3$-dimensional sundials.
   Cent. Eur. J. Math. 9(5) (2011), 949--971

\bibitem{CCG13}
   Carlini, E., Catalisano, M. V., Geramita, A. V.:
   On the Hilbert function of lines union one non-reduced point.
   Ann. Sc. Norm. Super. Pisa Cl. Sci. (5) Vol. XV (2016), 69--84

\bibitem{DHST14}
   Dumnicki M., Harbourne, B., Szemberg T., Tutaj-Gasi\'nska H.:
   Linear subspaces, symbolic powers and Nagata type conjectures.
   Adv. Math. 252 (2014), 471--491

\bibitem{Har77}
   Hartshorne, R.:
   Algebraic geometry.
   Graduate Texts in Mathematics, No. 52. Springer-Verlag 1977

\bibitem{HarHir81}
   Hartshorne, R., Hirschowitz, A.:
   Droites en position g\'en\'erale dans l'espace projectif.
   Algebraic geometry (La R\'abida, 1981), 169--188, Lecture Notes in Math., 961, Springer-Verlag 1982

\bibitem{PAG}
   Lazarsfeld, R.:
   Positivity in algebraic geometry.
   Ergebnisse der Mathematik und ihrer Grenzgebiete. 3. Folge.
   A Series of Modern Surveys in Mathematics 48, 49. Springer-Verlag, Berlin 2004.

\bibitem{Len11}
   Lenarcik, T.:
   Linear systems over $\P^1\times\P^1$ with base points of multiplicity bounded by three.
   Ann. Polon. Math. 101 (2011), 105--122

\end{thebibliography}
\end{document}